\documentclass[1996/10/24]{amsart}
\usepackage{color}
\usepackage{amssymb}

\newtheorem{theorem}{Theorem}[section]
\newtheorem{lemma}[theorem]{Lemma}

\theoremstyle{definition}

\newtheorem{example}[theorem]{Example}

\theoremstyle{remark}
\newtheorem{remark}[theorem]{Remark}

\numberwithin{equation}{section}

\usepackage{color}
\definecolor{black}{rgb}{0,0,0}

\definecolor{red}{rgb}{1,0,0}

\definecolor{blue}{rgb}{0,0,1}

\def\Real{\mathbb{R}}

\def\bff{\textit{\textbf{f}}}

\def\bfg{\textit{\textbf{g}}}
\def\bfn{\textit{\textbf{n}}}

\def\bfr{\textit{\textbf{r}}}
\def\bfu{\textit{\textbf{u}}}
\def\bfv{\textit{\textbf{v}}}
\def\bfw{\textit{\textbf{w}}}

\def\bfe{\textit{\textbf{e}}}
\def\bfL{\textit{\textbf{L}}}
\def\bfH{\textit{\textbf{H}}}
\def\bfM{\textit{\textbf{M}}}

\def\bfP{\textit{\textbf{P}}}

\def\bfV{\textit{\textbf{V}}}
\def\bfW{\textit{\textbf{W}}}

\def\bfPi{\boldsymbol{\Pi}}

\def\bfmu{\boldsymbol{\mu}}

\def\bftheta{\boldsymbol{\theta}}
\def\bfeta{\boldsymbol{\eta}}

\def\bftheta{\boldsymbol{\theta}}
\def\bfphi{\boldsymbol{\phi}}
\def\bfpsi{\boldsymbol{\psi}}

\def\calT{\mathcal{T}}

\def\calE{\mathcal{E}}

\begin{document}

\title[A HDG method for Maxwell equations]
{A superconvergent  HDG method for the Maxwell equations}


\author{Huangxin Chen}
\address{School of Mathematical Sciences and Fujian Provincial Key Laboratory on Mathematical Modeling and 
High Performance Scientific Computing, Xiamen University, Fujian, 361005, China}
\email{chx@xmu.edu.cn}

\author{Weifeng Qiu}
\address{Department of Mathematics, City University of Hong Kong,
83 Tat Chee Avenue, Kowloon, Hong Kong, China}
\email{weifeqiu@cityu.edu.hk}

\author{Ke Shi}
\address{Department of Mathematics $\&$ Statistics, Old Dominion University, Norfolk, VA 23529, USA}
\email{shike1983@gmail.com}

\author{Manuel Solano}
\address{Departamento de Ingenier\'ia Matem\'atica 
and Centro de Investigaci\'on 
en Ingenier\'ia Matem\'atica (CI$^2$MA), Universidad de Concepci\'on, Concepci\'on, Chile}
\email{msolano@ing-mat.udec.cl}

\thanks{Corresponding author: Weifeng Qiu.}



\date{}

\begin{abstract}
We present and analyze a new hybridizable discontinuous Galerkin (HDG) method for the steady state Maxwell equations. In order to make the problem well-posed, a condition of divergence is imposed on the electric field. Then a Lagrange multiplier $p$ is introduced, and the problem becomes the solution of a mixed curl-curl formulation of the Maxwell's problem. We use polynomials of degree $k+1$, $k$, $k$ to approximate $\bfu,\nabla \times \bfu$ and $p$ respectively. In contrast, we only use a non-trivial subspace of polynomials of degree $k+1$ to approximate the numerical tangential trace of the electric field and polynomials of degree $k+1$ to approximate the numerical trace of the Lagrange multiplier on the faces. On the simplicial meshes, a special choice of the stabilization parameters is applied, and the HDG system is shown to be well-posed. Moreover, we show that the convergence rates for $\boldsymbol{u}$ and $\nabla \times \boldsymbol{u}$ are independent of the Lagrange multiplier $p$. If we assume the dual operator of the Maxwell equation on the domain has adequate  regularity, we show that the convergence rate for $\boldsymbol{u}$ is $O(h^{k+2})$. From the point of view of degrees of freedom of the globally coupled unknown: numerical trace, this HDG method achieves superconvergence for the electric field without postprocessing. Finally, we show that on general polyhedral elements, by a particular choice of the stabilization parameters again, the HDG system is also well-posed and the 
superconvergence of the HDG method is derived.
\end{abstract}

\maketitle

\section{Introduction}
In this paper, we consider the Maxwell equations of the following form:
\begin{subequations}\label{original_pde_0}
\begin{align}
\nabla \times (\nabla \times \bfu) + \nabla p &= \boldsymbol{f} \qquad   \text{in $\Omega$, }\\
\nabla \cdot \bfu & = 0  \qquad \ \text{in $\Omega$}, \\
\bfu \times \bfn & = \bfg  \qquad \ \text{on $\partial \Omega$}, \\
p &= 0 \qquad \ \text{on $\partial \Omega$},
\end{align}
\end{subequations}
where $\Omega$ is a simply connected polyhedral domain in $\mathbb{R}^3$ and the unknowns are $\bfu$ and $p$. Here $\bff$ is the source term and $\bfg$ is given on $\partial \Omega$. The above Maxwell's problem follows from finding an approximation of the following problem within a three dimensional domain:
\[
\nabla \times (\nabla \times \bfu) = \bff,\quad \nabla \cdot \bfu  = 0 , \quad \bfu \times \bfn|_{\partial \Omega}  = \bfg. 
\]
For simplicity, we use divergence free condition of the electric field. Actually, we can also assume $\nabla \cdot \bfu  = \sigma \in L^2(\Omega)$ and $\sigma$ is nonzero. In order to have a better control on the divergence of the electric field $\bfu$, a Lagrange multiplier $p$ is introduced (cf. \cite{Bonito2013,Bonito2014,MWYS15}), and we consider a generalized mixed formulation (\ref{original_pde_0}). It is not necessary to require that the source term $\bff$ is divergence free. Actually, the Lagrange multiplier $p$ accommodates the possible non-zero divergence of $\bff$. Obviously, $p = 0$ if $\bff$ is divergence free.

Various numerical methods have been studied to solve the Maxwell's problem in the literature which include $H({\rm curl},\Omega)$-conforming edge element methods \cite{Nedelec80,Nedelec86,Hiptmair02,Monk2003,Zhong09}, discontinuous Galerkin (DG) methods \cite{Perugia02,Perugia03,Cockburn04,Houston04,Houston05,Brenner2008,Nguyena11,R.Hiptmair13,FW2014}, interior penalty method with $C^0$ finite element \cite{Bonito2014} and weak Galerkin FEM method \cite{MWYS15}. The DG methods have several attractive features which include the capabilities to handle complex geometries, to provide 
high-order accurate solutions, etc. For instance, the mixed DG formulation for the problem (\ref{original_pde_0}) was proposed and analyzed. However, the dimension of the standard approximation DG space is much larger than the 
dimension of the corresponding conforming space. Hybridizable discontinuous Galerkin (HDG) methods \cite{Cockburn09} 
were recently introduced to address this issue. The HDG methods retain the advantages of standard DG methods, 
and the resulting system is only due to the unknowns on the skeleton of the mesh. The HDG methods were introduced in \cite{Nguyena11} for the numerical solution of the time-harmonic Maxwell problem. Compared to the IPDG method for the  Maxwell equations in \cite{Houston04,Houston05,FW2014}, the HDG methods have less globally coupled unknowns.

In order to apply HDG method, as usual, we first write the Maxwell equations (\ref{original_pde_0}) into a system of first order equations. To this end, we introduce a new unknown $\bfw = \nabla \times \bfu$. Now we can write the Maxwell equations (\ref{original_pde_0}) in a mixed curl-curl formulation as follows: 
\begin{subequations}\label{original}
\begin{align}
\label{original_1}
 \nabla \times \bfu  & = \bfw \qquad \text{in $\Omega$}, \\
\label{original_2}
\nabla \times \bfw + \nabla p &= \boldsymbol{f} \qquad   \text{in $\Omega$, }\\
\label{original_3}
\nabla \cdot \bfu & = 0 \qquad \ \text{in $\Omega$}, \\
\label{original_4}
\bfu \times \bfn & = \bfg \qquad \ \text{on $\partial \Omega$}, \\
\label{original_5}
p &= 0 \qquad \ \text{on $\partial \Omega$}.
\end{align}
\end{subequations}

The objective of this paper is to develop a new HDG method which is absolutely well-posed for the above mixed curl-curl formulation (\ref{original}). By an appropriate choice of numerical trace of the electric field on the faces, we prove the well-posedness of the HDG system and analyze its convergence property. The enhanced space for the primary variable of numerical trace was first introduced in \cite{Lehrenfeld10} for diffusion problem which numerically showed that the methods provide optimal order of convergence for all unknowns. Recently, the authors in \cite{QS13,QS16_1,QS16_2} applied this approach for linear elasticity problems, convection-diffusion problems and incompressible Navier-Stokes equations and gave the rigorous analysis. In this paper, we extend this approach to the Maxwell equations with two new HDG methods. We notice that the curl operator contains a non-trivial kernel space, which will be considered in designing the enhanced space for numerical trace of electric field on faces. We first present a HDG scheme based on simplicial mesh, the polynomials of degree $k+1$, $k$, $k$ ($k\geq 0$) are used to approximate $\bfu,\nabla \times \bfu$ and $p$ respectively. An enhanced subspace of polynomials of degree $k+1$ is used to approximate the numerical trace of the electric field on the faces and polynomials of degree $k+1$ are used to approximate the numerical trace of the Lagrange multiplier on the faces. By carefully designing the numerical flux, we obtain optimal convergence for the discrete $H({\rm curl})$-norm error and discrete $H({\rm div})$-norm error in electric field $\bfu$, and the $L^2$-norm error in $\bfw$. The duality argument is applied to derive the optimal convergence for the $L^2$-norm error in the electric field. In addition, thanks to the specially chosen stabilization parameters, we can obtain globally divergence free approximation for $\boldsymbol{u}$ and the errors for $\boldsymbol{u}$ and $\boldsymbol{w}$ are independent of the Lagrange multiplier $p$. Up to our best knowledge, this is the first error estimate for the mixed curl-curl formulation of the Maxwell's problem (\ref{original_pde_0}) to be independent of the Lagrange multiplier $p$. Furthermore, we consider a modified scheme by adding an additional stabilization term in the flux so that the method can maintain all convergence results mentioned above. Similar as our work in \cite{QS13,QS16_1,QS16_2}, the analysis is valid for general polyhedral meshes. Nevertheless, the errors do depend on the regularity of the pseudo variable $p$. Finally it is worth to mention that both methods are hybridizable in the sense that the global unknowns are the numerical traces. From the point of view of global degrees of freedom, this HDG method provides optimal convergent approximations to the electric field and achieves superconvergence for the electric field without postprocessing.

The rest of paper is organized as follows. In Section \ref{section2}, we introduce our HDG method for the problem (\ref{original}) and illustrate the well-posedness of both methods. In Section \ref{section3}, we present the adjoint problem and give the main a priori error estimates. The detailed proof of the main results is provided in Section \ref{section4} which includes both cases for the HDG schemes on simplicial mesh and general polyhedral mesh.

\section{Preliminaries and the HDG method}\label{section2}

In this section we will present the detail formulation for both methods. From now on we use $\text{HDG}_s$ to denote the first method (on simplicial mesh) and $\text{HDG}_g$ to denote the second method (on general polyhedral mesh). 
To define the $\text{HDG}_s$ methods, we consider conforming triangulation $\calT_h$ of $\Omega$ made of shape-regular {simplicial elements}. 
We denote by $\calE_h$ the set of all faces $F$ of all elements $K\in \calT_h$ 
and set $\partial \calT_h:=\{\partial K:K\in \calT_h\}$.
For scalar-valued functions $\phi$ and $\psi$, we write
\[
 (\phi,\psi)_{\calT_h}:=\sum_{K\in\calT_h}(\phi,\psi)_K, \, \, \langle\phi,\psi\rangle_{\partial\calT_h}
 :=\sum_{K\in\calT_h}\langle\phi,\psi \rangle_{\partial K}.
\]
Here $(\cdot,\cdot)_D$ denotes the integral over
the domain $D\subset \Real^d$, and $\langle \cdot,\cdot \rangle_D$ denotes the
integral over $D\subset \Real^{d-1}$. For vector-valued and
matrix-valued functions, a similar notation is taken. For example, for
vector-valued functions, we write
$(\bfphi,\bfpsi)_{\calT_h}:=\sum_{i=1}^n(\phi_i,\psi_i)_{\calT_h}$. Finally, for a vector-valued function $\bfphi$, on any interface $F \in \calE_h$ we use $\bfphi^n, \bfphi^t$ to denote the normal and tangential components of $\bfphi$ respectively. Throughout the paper we use the standard notations and definitions for Sobolev spaces (see, e.g., Adams \cite{Adams}).

Like all other HDG schemes, to define the HDG method for 
the problem, we need to introduce some new unknowns called \emph{numerical traces} on the \emph{skeleton} of the mesh, $\calE_h$. In the case of Maxwell's equations \eqref{original}, we need to impose two numerical traces into our HDG formulation. Namely, our HDG method seeks an approximation 
$(\bfw_h, \bfu_h, p_h, \widehat{\bfu}^t_h, \widehat{p}_h)\in \bfW_h \times \bfV_h \times Q_h \times
\bfM^{t,\bfg}_h \times M^0_h$ to the exact solution 
$(\bfw|_{\mathcal{T}_h}, \bfu|_{\mathcal{T}_h}, p|_{\mathcal{T}_h},\bfu^t|_{\mathcal{E}_h}, p|_{\calE_h})$
where the finite dimensional spaces are defined as: $(k \ge 0)$
\begin{alignat*}{3}
 \bfW_h&:=\{\bfr \in \bfL^2(\Omega): \;\bfr|_K\in \bfP_k(K), \,\, \forall K\in \calT_h\},\\
 \bfV_h&:=\{\bfv \in \boldsymbol{L}^2(\Omega):\;\bfv|_K\in \bfP_{k+1}(K), \,\, \forall K\in \calT_h\},\\
 Q_h&:=\{q \in L^2(\Omega):\;q|_K \in P_{k}(K), \,\, \forall K\in \calT_h\},\\
 \bfM^{t}_h&:=\{\bfmu \in \boldsymbol{L}^2(\calE_h):\;\bfmu|_F\in \bfM^t(F), \, \forall F\in \calE_h \},\\
  \bfM^{t, \bfg}_h&:=\{\bfmu \in  \bfM^{t}_h:\, \bfmu\times \bfn|_{\partial \Omega} = \bfP_{\bfM} \bfg \},\\
 M_h^{0}&:=\{\zeta \in L^2(\calE_h):\; \zeta|_F \in P_{k+1}(F), \,\, \forall F \in \calE_h, \,\, \zeta|_{\partial \Omega}=0\},
\end{alignat*}
where  $\boldsymbol{L}^2(\Omega) = [L^2(\Omega)]^3$, $\boldsymbol{L}^2(\calE_h) = [L^2(\calE_h)]^3$, for each $F \in \calE_h$, the local space $\bfM^t(F)$ is defined as
$\bfM^t(F):= [\bfP_k(F) + \nabla \widetilde{P}_{k+2}(F)]^t$, $\widetilde{P}_{k+2}(F)$ denotes the set of homogeneous polynomials of degree $k+2$ on $F$. Obviuously, $\bfM^t(F) \subset \bfP_{k+1}(F)$. Here, $\bfP_{\bfM}$ denotes the orthogonal $L^2$-projection from $L^2(\calE_h)$ onto $\bfM^t_h$, $P_k(D)$ denotes the set of polynomials of total degree at most $k \geq 0$
defined on $D$, $\bfP_k(D)$ denotes the set of vector-valued functions whose $d$
components lie in $P_k(D)$. We seek an approximation $(\bfw_h, \bfu_h, p_h, \widehat{\bfu}^t_h, \widehat{p}_h)\in \bfW_h \times \bfV_h \times Q_h \times
\bfM^{t,\bfg}_h \times M^0_h$ satisfying:
\begin{subequations}\label{HDG_form}
\begin{align}
\label{HDG_form_1}
(\bfw_h, \bfr)_{\calT_h} - (\bfu_h , \nabla \times \bfr)_{\calT_h} + \langle \widehat{\bfu}^t_h \times \bfn \, , \bfr 
\rangle_{\partial \calT_h} & = 0, \\
\label{HDG_form_2}
( \bfw_h, \nabla \times \bfv)_{\calT_h} - (p_h, \nabla \cdot \bfv)_{\calT_h} + \langle \widehat{\bfw}_h, \bfv \times \bfn 
\rangle_{\partial \calT_h} + \langle \widehat{p}_h, \bfv \cdot \bfn \rangle_{\partial \calT_h} & = (\bff, \bfv)_{\calT_h}, \\
\label{HDG_form_3}
-(\bfu_h, \nabla q)_{\calT_h} + \langle \widehat{\bfu}_h^n \cdot \bfn \, , q  \rangle_{\partial \calT_h} &  = 0,\\
\label{HDG_form_4}
\langle \widehat{\bfw}_h \times \bfn \, , \bfeta \rangle_{\partial \calT_h} & = 0, \\
\label{HDG_form_5}
\langle \widehat{\bfu}^n_h \cdot \bfn \, , \zeta \rangle_{\partial \calT_h}  & = 0, \\
\intertext{for all $(\bfr, \bfv, q, \bfeta, \zeta) \in \bfW_h \times \bfV_h \times Q_h \times \bfM^{t,0}_h \times M^0_h$, 
where the \emph{numerical flux} is defined as}
\label{HDG_trace_w}
\widehat{\bfw}_h = \bfw_h + \tau_t (\bfP_{\bfM} \bfu^t_h - \widehat{\bfu}^t_h) \times \bfn,  \qquad\qquad\\
\label{HDG_trace_u}
\widehat{\bfu}^n_h \cdot \bfn = \bfu_h \cdot \bfn. \qquad \qquad\qquad \
\end{align}
\end{subequations}

The \emph{stabilization parameters} $\tau_t$ is defined on each $F \in \partial \calT_h$. For the sake of simplicity, we assume that $\tau_t$ is a non-negative constant which will be determined later from the analysis. 

To obtain the second method $\text{HDG}_g$, we only need to make two changes: (1) $\mathcal{T}_h$ can be any conforming polyhedral triangulation; (2) Adding another stabilization term in the the last governing equation \eqref{HDG_trace_u}:
\begin{equation}\label{new_trace_u}
\widehat{\bfu}^n_h \cdot \bfn = \bfu_h \cdot \bfn + \tau_n (p_h - \widehat{p}_h). 
\end{equation}
In this case we require that $\tau_n$ is non-negative on each interface $F \in \partial \mathcal{T}_h$. 

For the $\text{HDG}_s$ method, the governing equation \eqref{HDG_trace_u} and \eqref{HDG_form_5} implies that $\boldsymbol{u}_h \in \boldsymbol{H}({\rm div}, \Omega)$. In addition, if we apply integrating by parts on \eqref{HDG_form_3} and take $q = \nabla \cdot \boldsymbol{u}_h$, we can conclude that  
\begin{equation}\label{divfree_u}
\nabla \cdot \boldsymbol{u}_h = 0.
\end{equation}
In addition to this globally divergence free property, it can be shown that the error estimates are independent of the pseudo variable $p$. More details will be discussed in Section 4.

As a build-in feature of the HDG methods (cf. \cite{Cockburn09}), the unknowns $\bfw_h, \bfu_h, p_h$ in (\ref{HDG_form}) can be eliminated to obtain a formulation only with $\widehat{\bfu}^t_h$ and $\widehat{p}_h$  as global unknowns. It is worth to mention that on general polyhedral meshes, it is necessary to use the enhanced numerical flux \eqref{new_trace_u} in order to ensure the solvability of the system. Consequently, $\boldsymbol{u}_h$ is no longer globally divergence free.

To end this section, we present the well-posedness result of the both scheme which can be summarized as follows:

\begin{lemma}\label{well-posedness}
If the stabilization parameter $\tau_t$ is positive everywhere, for $k \ge 0$, the $\text{HDG}_s$ system \eqref{HDG_form} is well-posed, and the approximation $\boldsymbol{u}_h$ is globally divergence free.

In addition, if $\mathcal{T}_h$ is any conforming polyhedral triangulation and $\tau_t > 0, \tau_n > 0$, then $\text{HDG}_g$ system \eqref{HDG_form_1} - \eqref{HDG_trace_w}, \eqref{new_trace_u} is also well-posed.  
\end{lemma}

\begin{proof}
The divergence free property of $\boldsymbol{u}_h$ is proved in the above discussion. Since \eqref{HDG_form} is a linear square system, it suffices to show that if all the given data vanish, i.e. $\bff = \bfg = 0$, then we will get zero solution for all unknowns. By taking $(\bfr, \bfv, q, \bfeta, \zeta) = (\bfw_h, \bfu_h, p_h, \widehat{\bfu}^t_h, \widehat{p}_h)$ in the equations \eqref{HDG_form_1}-\eqref{HDG_form_5} and adding these equations, after some algebraic manipulation, we have:
\[
(\bfw_h, \bfw_h)_{\calT_h} + \langle \tau_t (\bfP_{\bfM}{\bfu}^t_h \times \bfn  - \widehat{\bfu}^t_h \times \bfn) , {\bfu}_h \times \bfn - \widehat{\bfu}^t_h \times \bfn \rangle_{\partial \calT_h} = 0.
\]
Immediately we obtain:
\[
\bfw_h = 0 \quad \text{and}  \quad \bfP_{\bfM} {\bfu}^t_h - \widehat{\bfu}^t_h = 0 \quad \text{on $\partial \calT_h$}.
\]

\[
(\bfw_h, \bfw_h)_{\calT_h} - \langle \bfu_h \cdot \bfn - \widehat{\bfu}^n_h \cdot \bfn \, , p_h - \widehat{p}_h \rangle_{\partial \calT_h} + \langle \widehat{\bfu}^t_h \times \bfn - \bfu_h \times \bfn \, , \bfw_h - \widehat{\bfw}_h \rangle_{\partial \calT_h} = 0.
\]
Now by applying the definition of the numerical flux \eqref{HDG_trace_w} and \eqref{HDG_trace_u}, we have:
\begin{equation}\label{energy_id}
(\bfw_h, \bfw_h)_{\calT_h} + \langle \tau_t (\bfP_{\bfM}{\bfu}^t_h \times \bfn  - \widehat{\bfu}^t_h \times \bfn) , {\bfu}_h \times \bfn - \widehat{\bfu}^t_h \times \bfn \rangle_{\partial \calT_h} = 0.
\end{equation}

It is clear that 
$$
\langle \tau_t (\bfP_{\bfM}{\bfu}^t_h \times \bfn  - \widehat{\bfu}^t_h \times \bfn) , {\bfu}_h \times \bfn - \widehat{\bfu}^t_h \times \bfn \rangle_{\partial \calT_h} = \langle \tau_t (\bfP_{\bfM}{\bfu}^t_h  - \widehat{\bfu}^t_h ), \bfP_{\bfM}{\bfu}^t_h - \widehat{\bfu}^t_h  \rangle_{\partial \calT_h}.
$$
Since we assume that $\tau_t$ is strictly positive on $\partial \calT_h$, the above two identities imply that 
\[
\bfw_h = 0 \quad \text{and} \quad \bfP_{\bfM} {\bfu}^t_h - \widehat{\bfu}^t_h = 0 \quad \text{on $\partial \calT_h$}.
\]
By these results together with \eqref{HDG_trace_w} we have
\[
\widehat{\bfw}_h = 0.
\]
Next we rewrite \eqref{HDG_form_1} by integrating by parts on the second term to have:
\[
(\bfw_h, \bfr)_{\calT_h} - (\nabla \times \bfu_h, \bfr)_{\calT_h} - \langle (\bfu_h - \widehat{\bfu}^t_h) \times \bfn \, , \bfr \rangle_{\partial \calT_h} = 0.
\]
If we take $\bfr = \nabla \times \bfu_h$, by the above result and the orthogonal property of $L^2$-projection $\bfP_{\bfM}$, we have:
\[
\nabla \times \bfu_h = 0.
\]
Since $\bfw_h = 0, \widehat{\bfw}_h = 0$, now \eqref{HDG_form_2} can be written as (given $\bff = 0$):
\begin{equation}\label{eq_p}
-(p_h, \nabla \cdot \bfv)_{\calT_h} + \langle \widehat{p}_h, \bfv \cdot \bfn \rangle_{\partial \calT_h} = 0.
\end{equation}
By integrating by parts, we have
\begin{equation}\label{lifting_a}
(\nabla p_h, \bfv)_{\calT_h} + \langle \widehat{p}_h - p_h, \bfv \cdot \bfn \rangle_{\partial \calT_h} = 0, \quad \text{for all $\boldsymbol{v} \in \boldsymbol{V}_h$}.
\end{equation}
On each $K \in \mathcal{T}_h$, we choose $\boldsymbol{v} \in \boldsymbol{P}_{k+1}(K)$ to satisfy the following equations:
\begin{align*}
(\boldsymbol{v}, \tilde{\boldsymbol{v}})_K & = (\nabla p_h, \tilde{\boldsymbol{v}})_K, \quad && \text{for all $\tilde{\boldsymbol{v}} \in \boldsymbol{P}_k(K)$}, \\
\langle \boldsymbol{v} \cdot \bfn \, , \, \mu \rangle_F & = \langle  \widehat{p}_h - p_h \, , \, \mu \rangle_F, && \text{for all $\mu \in P_k(F), \, F \in \partial K$}.
\end{align*}
The existence of such $\boldsymbol{v}$ is ensured if $K$ is simplex. For instance, by the above degrees of freedom uniquely define a function $\boldsymbol{v} \in \bf{RT}_k(K) \subset \boldsymbol{P}_{k+1}(K)$. Here ${\bf RT}_k(K)$ denotes the Raviert-Thomas space of degree $k$ on the simplex $K$. By this special choice of $\boldsymbol{v}$, \eqref{lifting_a} gives
\[
(\nabla p_h, \nabla p_h)_{\calT_h} + \langle \widehat{p}_h - p_h, \widehat{p}_h - p_h \rangle_{\partial \calT_h} = 0
\]
This implies that 
\[
\nabla p_h = 0, \quad \widehat{p}_h - p_h = 0,
\]
and hence $p_h = \widehat{p}_h = 0$.

To conclude, recall that $\nabla \times \bfu_h = 0$ and $\bfu_h \in P_{k+1}(K)$ on each $K \in \calT_h$, this means that there exists a piecewise polynomial $\phi_h \in P_{k+2}(K)$ for each $K \in \calT_h$ such that 
\[
\bfu_h = \nabla \phi_h \quad \text{on each $K \in \calT_h$}.
\]
The above fact together with the definition of the space $\bfM^t_h$ implies that
\[
\bfP_{\bfM} (\bfu^t_h) = \bfP_{\bfM} ( \bfn \times\nabla \phi_h \times \bfn) = (\nabla \phi_h)^t = \bfu^t_h,
\]
since on each interface $F \in \calE_h$ we have $(\bfn \times \nabla \phi_h \times \bfn)|_F \in [\nabla P_{k+2}(F)]^t \subset \bfM^t(F)$. This implies that 
\[
\bfu^t_h= \widehat{\bfu}^t_h \quad {\rm on} \ \partial \calT_h, 
\]
which, together with the facts that 
\[
\bfu_h \cdot \bfn = \widehat{\bfu}^n_h \cdot \bfn \quad {\rm on} \ \partial \calT_h,
\]
and both $\widehat{\bfu}^t_h, \widehat{\bfu}^n_h \cdot \bfn$ are single values on $\calE_h$, implies that $\bfu_h \in \bfH(\text{curl}, \Omega) \cap \bfH(\text{div}, \Omega)$ satisfying:
\begin{align*}
\nabla \times \bfu_h &= 0 \qquad  \text{in $\Omega$,}\\
\nabla \cdot \bfu_h & = 0 \qquad  \text{in $\Omega$,} \\
\bfu_h \times \bfn & = 0 \qquad  \text{on $\partial \Omega$}.
\end{align*}
This yields $\bfu_h=0$ by the assumption of the domain $\Omega$ (cf. \cite{Monk2003}). Therefore, we also have $\widehat{\bfu}^t_h=0$ on $\partial \calT_h$. This completes the proof of the solvability of $\text{HDG}_s$.

For the general method $\text{HDG}_g$, we briefly sketch the proof in this case since it follows similar argument as the first case. First, due to the fact that the stabilization parameter $\tau_n$ is positive, the energy identity \ref{energy_id} becomes:
\begin{equation}\label{energy_id_new}
(\bfw_h, \bfw_h)_{\calT_h} + \langle \tau_n (p_h - \widehat{p}_h) \, , p_h - \widehat{p}_h \rangle_{\partial \calT_h} + \langle \tau_t (\bfP_{\bfM}{\bfu}^t_h \times \bfn  - \widehat{\bfu}^t_h \times \bfn) , {\bfu}_h \times \bfn - \widehat{\bfu}^t_h \times \bfn \rangle_{\partial \calT_h} = 0,
\end{equation}
which implies that
\[
(\bfw_h, \bfw_h)_{\calT_h} + \langle \tau_n (p_h - \widehat{p}_h) \, , p_h - \widehat{p}_h \rangle_{\partial \calT_h} + \langle \tau_t (\bfP_{\bfM}{\bfu}^t_h  - \widehat{\bfu}^t_h ), \bfP_{\bfM}{\bfu}^t_h - \widehat{\bfu}^t_h  \rangle_{\partial \calT_h} = 0.
\]
Therefore, we have
\[
\bfw_h = 0, \quad p_h - \widehat{p}_h = 0, \quad \text{and} \quad \bfP_{\bfM} {\bfu}^t_h - \widehat{\bfu}^t_h = 0 \quad \text{on $\partial \calT_h$}.
\]
And by \eqref{HDG_trace_w} we have
\[
\widehat{\bfw}_h = 0.
\]
Consequently, \eqref{eq_p} is still valid in this case. Taking $\bfv = \nabla p_h$ and integrating by parts for \eqref{eq_p} and by the fact that $p_h - \widehat{p}_h = 0$ on $\partial \calT_h$, we can obtain:
\[
\nabla p_h = 0 \quad \text{on each $K \in \calT_h$}.
\]
This implies that $p_h$ is piecewise constant on $\calT_h$. By the fact that $p_h = \widehat{p}_h$ on $\calE_h$, we conclude that $p_h = 0$ on $\Omega$, and hence $\widehat{p}_h = 0$. So far, we have shown that $\bfw_h, p_h, \widehat{p}_h$ vanish. 

Finally, taking $q = \nabla \cdot \bfu_h$ in \eqref{HDG_form_3} and integrating by parts, we have
\[
(\nabla \cdot \bfu_h, \nabla \cdot \bfu_h)_{\calT_h} - \langle \bfu_h \cdot \bfn - \widehat{\bfu}_h \cdot \bfn \, , \nabla \cdot \bfu_h \rangle_{\partial \calT_h} = 0. 
\]
Since that
\[
\bfu_h \cdot \bfn - \widehat{\bfu}_h \cdot \bfn = -\tau_n (p_h - \widehat{p}_h) = 0 \quad {\rm on} \ \partial \calT_h,
\]
we have
\[
\nabla \cdot \bfu_h = 0 \quad \text{in $\Omega$}.
\]
After this point, by the same lines as for the first case we can show $\nabla \times \boldsymbol{u}_h = 0$ and then conclude that $\boldsymbol{u}_h, \widehat{\boldsymbol{u}}^t_h$ also vanish. This completes the proof for the general case. 
\end{proof}

\section{Main results}\label{section3}
In this section we present our main error estimates results. To state our main result, we need to introduce some notations.
We use $\|\cdot\|_{s, D}, |\cdot|_{s, D}$ to denote the usual norm and semi-norm on
the Sobolev space $H^s(D)$. We discard the first index $s$ if $s=0$. The norm $\|\cdot\|_{s, \calT_h}$ is the discrete norm defined as $\|\cdot\|_{s, \calT_h}:= \sum_{K \in \calT_h}\|\cdot\|_{s, K}$. We also define the discrete norm $\|\cdot\|_{\partial \calT_h}:= \sum_{F \in \partial \calT_h} \|\cdot\|_F$.
To derive an $L^2$ error estimate of $\|\bfu - \bfu_h\|_{\Omega}$, we need a regularity assumption of an adjoint problem stated as follows: Find $(\bfphi, \bftheta, \sigma)$ satisfying

\begin{subequations}\label{adj_eq}
\begin{align}
\label{adj_1} 
\nabla \times \bfphi & = \bftheta \qquad \text{in $\Omega$}, \\
\label{adj_2} 
\nabla \times \bftheta - \nabla \sigma & = \bfe_u \quad \ \,  \text{in $\Omega$}, \\
\label{adj_3} 
\nabla \cdot \bfphi & = 0 \qquad   \text{in $\Omega$}, \\
\label{adj_4} 
\bfphi \times \bfn & = 0 \qquad \text{on $\partial \Omega$}, \\ 
\label{adj_5} 
\sigma & = 0 \qquad \text{on $\partial \Omega$}.
\end{align}
\end{subequations}
We assume the following regularity estimate holds for the solution of the above adjoint problem (cf. \cite{Bonito2013,Bonito2014,MWYS15}):
\begin{equation}\label{regularity}
\|\bfphi\|_{1+\alpha, \Omega} + \|\bftheta\|_{\alpha, \Omega} + \|\sigma\|_{\alpha, \Omega} \le C_{\text{reg}} \|\bfe_u\|_{\Omega},
\end{equation}
where $0<\alpha\leq 1$.

We are now ready to state our first main result for $\text{HDG}_s$ method.

\begin{theorem}\label{main_mod}
If the exact solution of (\ref{original}) satisfies $(\bfu, \bfw, p) \in [H^t(\Omega)]^3 \times [H^s(\Omega)]^3 \times [H^1_0(\Omega)\cap H^r(\Omega)]$ 
with $1 \le s \le k+1, \, 1 \le r \le k+1,\, 1 \le t \le k+2$. and $\tau_t = \mathcal{O}(\frac{1}{h})$, then for the numerical solution of the $\text{HDG}_s$ method, we have 
\begin{align*}
&\|\nabla \cdot (\bfu - \bfu_h)\|_{\calT_h} = 0, \\
&\|\bfw - \bfw_h\|_{\Omega} + \|\nabla \times (\bfu - \bfu_h)\|_{\calT_h} \le  C \Big( h^{t-1} \|\bfu\|_{t, \Omega} + h^s \|\bfw\|_{s, \Omega} \Big),
\end{align*}
for all $1 \le s \le k+1, 1 \le t \le k+2$. 

In addition, if the regularity \eqref{regularity} holds for the adjoint problem, then we have
\[
\|\bfu-\bfu_h\|_{\Omega} \le C \Big( h^{t+\alpha-1} \|\bfu\|_{t, \Omega} + h^{s+\alpha} \|\bfw\|_{s, \Omega} \Big),
\]
where $0<\alpha\leq 1$ is defined in \eqref{regularity}.
\end{theorem}

As for the $\text{HDG}_g$ method on general polyhedral mesh, we have the following result:

\begin{theorem}\label{main_result}
If the exact solution of (\ref{original}) satisfies $(\bfu, \bfw, p) \in [H^t(\Omega)]^3 \times [H^s(\Omega)]^3 \times [H^1_0(\Omega)\cap H^r(\Omega)]$ 
with $1 \le s \le k+1, \, 1 \le r \le k+1,\, 1 \le t \le k+2$. Taking the stabilization parameters as $\tau_t = \mathcal{O}(\frac{1}{h}), \tau_n = \mathcal{O}(h)$, we have
\[
\|\bfw - \bfw_h\|_{\Omega} + \|\nabla \cdot (\bfu - \bfu_h)\|_{\calT_h} + \|\nabla \times (\bfu - \bfu_h)\|_{\calT_h} \le  C \Big( h^{t-1} \|\bfu\|_{t, \Omega} + h^{r}\|p\|_{r, \Omega} + h^s \|\bfw\|_{s, \Omega} \Big).
\]
In addition, if the regularity \eqref{regularity} holds for the adjoint problem, then we have
\[
\|\bfu-\bfu_h\|_{\Omega} \le C \Big( h^{t+\alpha-1} \|\bfu\|_{t, \Omega} + h^{r+\alpha} \|p\|_{r,\Omega} + h^{s+\alpha} \|\bfw\|_{s, \Omega} \Big),
\]
where $0<\alpha\leq 1$ is defined in \eqref{regularity}.
\end{theorem}

\begin{remark}
When the exact solution $(\bfu, \bfw, p)$ of (\ref{original}) is smooth enough, that is, $(\bfu, \bfw, p) \in [H^{k+2}(\Omega)]^3 \times [H^{k+1}(\Omega)]^3 \times [H^1_0(\Omega)\cap H^{k+1}(\Omega)]$, one can directly derive $\|\bfu-\bfu_h\|_{\Omega} \le C h^{k+1+\alpha}$. In particular, if $\alpha = 1$, we have $\|\bfu-\bfu_h\|_{\Omega} \le C h^{k+2}$. From the global degrees of freedom viewpoint, this error estimate reveals that both HDG methods achieve superconvergence for the electric field without postprocessing. In addition, the error estimates for $\boldsymbol{u}, \boldsymbol{w}$ in Theorem \ref{main_mod} are independent of the Lagrange multiplier $p$. The main ingredient for the proof is the use of the {\bf BDM} projection instead of $L^2$-projection for the variable $\boldsymbol{u}$.
See Section 4 for more details.
\end{remark}

\section{Proofs of the error estimates}\label{section4}
In this section we present the detail proof of Theorem \ref{main_mod}. Then we briefly sketch the proof of Theorem \ref{main_result} by particularly pointing out the different steps from the proof of Theorem \ref{main_mod}. 

\subsection{Error estimates for $\text{HDG}_s$ method on simplicial mesh} 
We begin by introducing the error equations that we need for the error estimates. We define the projection of the errors as below:
\begin{align*}
\bfe_w &:= \bfPi_W \bfw - \bfw_h, \quad \bfe_u := \bfPi^{k+1}_{\bf BDM}  \bfu - \bfu_h, \quad e_p := \Pi_Q p - p_h, \\
\bfe_{\widehat{u}^t} & := \bfP_{\bfM} \bfu^t - \widehat{\bfu}^t_h, \quad e_{\widehat{p}} := P_M p - \widehat{p}_h.
\end{align*}

Here all the projections $\bfPi_W, \Pi_Q, \bfP_M, P_M$ are the standard $L^2$-projections on the finite dimensional spaces $\bfW_h, Q_h, \bfM_h^t, M_h$ respectively. The projection $\bfPi^{k+1}_{\bf BDM} \boldsymbol{u}$ is the standard {\bf BDM} projection of $\boldsymbol{u}$ in $\bfP_{k+1}(K)$, for all $K \in \mathcal{T}_h$. In the analysis, we will use the following standard approximation properties of these projections:

\begin{subequations}\label{classic_ineq}
\begin{align}
\label{classic_1}
\|\bfu - \bfPi^{k+1}_{\bf BDM} \bfu \|_{\calT_h} & \le C h^{t-\delta} \|\bfu\|_{t, \Omega}, && 0 \le t \le k+2,\, \delta = 0,1,  \\
\label{classic_2}
\|\bfu - \bfPi^{k+1}_{\bf BDM} \bfu \|_{\partial \calT_h} & \le C h^{t- \frac12} \|\bfu\|_{t, \Omega},  && 1 \le t \le k+2, \\
\label{classic_3}
\|\bfw - \bfPi_W \bfw\|_{\Omega} &\le C h^s \|\bfw\|_{\Omega}, && 0 \le s \le k+1,\\
\label{classic_4}
\|\bfw - \bfPi_W \bfw\|_{\partial \calT_h} &\le C h^{s-\frac12} \|\bfw\|_{s, \Omega}, && 1 \le s \le k+1, \\
\label{classic_5}
\|p - \Pi_Q p\|_{\Omega} &\le C h^r \|p\|_{r, \Omega}, && 0 \le r \le k+1, \\
\label{classic_6}
\|p - \Pi_Q p\|_{\partial \calT_h} &\le C h^{r-\frac12} \|p\|_{r, \Omega}, && 1 \le r \le k+1, \\
\label{classic_7}
\|p - P_M p\|_{\partial \calT_h} &\le C h^{r-\frac12} \|p\|_{r, \Omega}, && 1 \le r \le k+2.
\end{align}
\end{subequations}

We are now ready to present the error equations that we need for the error estimates.
\begin{lemma}\label{error_eq}
Let $(\bfw,\bfu,p)$ and $(\bfw_h, \bfu_h, p_h, \widehat{\bfu}^t_h, \widehat{p}_h)$ solve the equations (\ref{original_pde_0})
and (\ref{HDG_form}), we have
\begin{subequations}\label{error_eqs}
\begin{align}
\label{error_1}
(\bfe_w, \bfr)_{\calT_h} - (\bfe_u , \nabla \times \bfr)_{\calT_h} + \langle \bfe_{\widehat{\bfu}^t} \times \bfn \, , \bfr \rangle_{\partial \calT_h} & = 0, \\
\label{error_2}
( \bfe_w, \nabla \times \bfv)_{\calT_h} - (e_p, \nabla \cdot \bfv)_{\calT_h} + \langle \bfw - \widehat{\bfw}_h, \bfv \times \bfn \rangle_{\partial \calT_h} + \langle e_{\widehat{p}}, \bfv \cdot \bfn \rangle_{\partial \calT_h} & = 0, \\
\label{error_3}
-(\bfe_u, \nabla q)_{\calT_h} + \langle \bfu \cdot \bfn - \widehat{\bfu}_h^n \cdot \bfn \, , q  \rangle_{\partial \calT_h} &  = 0,\\
\label{error_4}
\langle \bfw \times \bfn - \widehat{\bfw}_h \times \bfn \, , \bfeta \rangle_{\partial \calT_h} & = 0, \\
\label{error_5}
\langle \bfu \cdot \bfn - \widehat{\bfu}^n_h \cdot \bfn \, , \zeta \rangle_{\partial \calT_h} & = 0, 
\end{align}
\end{subequations}
for all $(\bfr, \bfv, q, \bfeta, \zeta) \in \bfW_h \times \bfV_h \times Q_h \times \bfM^{t,0}_h \times M^0_h$.
\end{lemma}
\begin{proof}
It is obvious that the exact solution $(\bfw, \bfu, p, \bfu^t|_{\calE_h}, p|_{\calE_h})$ also satisfies the system \eqref{HDG_form}:
\begin{align*}
(\bfw, \bfr)_{\calT_h} - (\bfu , \nabla \times \bfr)_{\calT_h} + \langle {\bfu}^t \times \bfn \, , \bfr \rangle_{\partial \calT_h} & = 0, \\
( \bfw, \nabla \times \bfv)_{\calT_h} - (p, \nabla \cdot \bfv)_{\calT_h} + \langle {\bfw}, \bfv \times \bfn \rangle_{\partial \calT_h} + \langle {p}, \bfv \cdot \bfn \rangle_{\partial \calT_h} & = (\bff, \bfv)_{\calT_h}, \\
-(\bfu, \nabla q)_{\calT_h} + \langle {\bfu}^n \cdot \bfn \, , q  \rangle_{\partial \calT_h} &  = 0,\\
\langle {\bfw} \times \bfn \, , \bfeta \rangle_{\partial \calT_h} & = 0, \\
\langle {\bfu}^n \cdot \bfn \, , \zeta \rangle_{\partial \calT_h} & = 0,
\end{align*}
 for all $(\bfr, \bfv, q, \bfeta, \zeta) \in \bfW_h \times \bfV_h \times Q_h \times \bfM^{t,0}_h \times M^0_h$. By the orthogonality of the projections as well as the inclusion properties between the finite dimensional spaces, we can write the above system as:
 \begin{align*}
(\bfPi_W \bfw, \bfr)_{\calT_h} - (\bfPi^{k+1}_{\bf BDM} \bfu , \nabla \times \bfr)_{\calT_h} + \langle \bfP_{\bfM} {\bfu}^t \times \bfn \, , \bfr \rangle_{\partial \calT_h} & = 0, \\
( \bfPi_W \bfw, \nabla \times \bfv)_{\calT_h} - (\Pi_Q p, \nabla \cdot \bfv)_{\calT_h} + \langle {\bfw}, \bfv \times \bfn \rangle_{\partial \calT_h} + \langle P_M {p}, \bfv \cdot \bfn \rangle_{\partial \calT_h} & = (\bff, \bfv)_{\calT_h}, \\
-(\bfPi^{k+1}_{\bf BDM} \bfu, \nabla q)_{\calT_h} + \langle {\bfu}^n \cdot \bfn \, , q  \rangle_{\partial \calT_h} &  = 0,\\
\langle {\bfw} \times \bfn \, , \bfeta \rangle_{\partial \calT_h} & = 0, \\
\langle {\bfu}^n \cdot \bfn \, , \zeta \rangle_{\partial \calT_h} & = 0,
\end{align*}
 for all $(\bfr, \bfv, q, \bfeta, \zeta) \in \bfW_h \times \bfV_h \times Q_h \times \bfM^{t,0}_h \times M^0_h$. Finally we obtain the error equations by subtracting \eqref{HDG_form} from the above equations. 
\end{proof}

We first present the estimate for $\bfe_w$ which can be achieved by several steps. First we carry out an important identity by a standard energy argument.

\begin{lemma}\label{lemma_energy_id}
We have
\begin{align*}
&\quad  \|\bfe_w\|^2_{\Omega} + \tau_t \|\bfP_{\bfM} \bfe_u^t - \bfe_{\widehat{u}^t}\|^2_{\partial \calT_h} \\
&=  - \langle \bfw - \bfPi_W {\bfw} \, , \, (\bfe_u^t - \bfe_{\widehat{u}^t}) \times \bfn \rangle_{\partial \calT_h}   - \langle \tau_t \bfP_{\bfM}(\bfu^t - (\bfPi^{k+1}_{\bf BDM} \bfu)^t) \times \bfn \, , \, (\bfe_u^t - \bfe_{\widehat{u}^t}) \times \bfn \rangle_{\partial \calT_h} \\
&:= - T_1 - T_2.
\end{align*}
\end{lemma}
\begin{proof}
Taking $(\bfr, \bfv, q, \bfeta, \zeta) = (\bfe_w, \bfe_u, e_p, \bfe_{\widehat{u}^t}, e_{\widehat{p}})$ in the error equations \eqref{error_1}-\eqref{error_5} and adding, after some algebraic manipulations, we obtain:
\begin{equation}\label{energy_1}
(\bfe_w, \bfe_w)_{\calT_h} + \langle \bfw - \widehat{\bfw}_h -  \bfe_w \, , \, \bfe_u^t \times \bfn - \bfe_{\widehat{u}^t} \times \bfn \rangle_{\partial \calT_h} + \langle \bfu \cdot \bfn - \widehat{\bfu}^n_h \cdot \bfn - \bfe_u \cdot \bfn \, , \, e_p - e_{\widehat{p}} \rangle_{\partial \calT_h} = 0.
\end{equation}
Notice that by the definition of the numerical traces \eqref{HDG_trace_w} and \eqref{HDG_trace_u} we can write
\begin{align*}
\bfw - \widehat{\bfw}_h - \bfe_w & = \bfw - \widehat{\bfw}_h - ( \bfPi_W \bfw - \bfw_h) = \bfw -\bfPi_W \bfw  + (\bfw_h - \widehat{\bfw}_h) \\
&=  \bfw -\bfPi_W \bfw - \tau_t (\bfP_{\bfM} \bfu^t_h - \widehat{\bfu}^t_h) \times \bfn \\
&=  \bfw -\bfPi_W \bfw + \tau_t (\bfP_{\bfM} \bfe_u^t - \bfe_{\widehat{\bfu}^t}) \times \bfn + \tau_t \bfP_{\bfM}(\bfu^t - (\bfPi^{k+1}_{\bf BDM} \bfu)^t) \times \bfn,\\
 \bfu \cdot \bfn - \widehat{\bfu}^n_h \cdot \bfn - \bfe_u \cdot \bfn & = \bfu \cdot \bfn - \widehat{\bfu}^n_h \cdot \bfn - (\bfPi^{k+1}_{\bf BDM} \bfu - \bfu_h) \cdot \bfn = (\bfu - \bfPi^{k+1}_{\bf BDM} \bfu) \cdot \bfn.
\end{align*} 
Inserting the above two expressions into \eqref{energy_1} together with the orthogonal property of the $L^2$-projection $\bfP_{\bfM}$ and $\bfPi^{k+1}_{\bf BDM}$ on the interfaces we complete the proof.
\end{proof}

Based on the identity in the above lemma, we are ready to estimate $\bfe_w$. 

\begin{lemma}\label{estimate_ew}
If the regularity of the exact solution $(\bfu, \bfw, p)$ of (\ref{original}) holds as assumed in Theorem \ref{main_mod}, the stabilization parameters $\tau_t = \mathcal{O}(\frac{1}{h})$ on each $F \in \partial \calT_h$, we have
\begin{align*}
\|\nabla \cdot (\bfu - \bfu_h)\|_{\calT_h} & = 0, \\
\quad \|\bfe_w\|_{\Omega} + \tau^{\frac12}_t \|\bfP_{\bfM} \bfe_u^t - \bfe_{\widehat{u}^t}\|_{\partial \calT_h} + \|\nabla \times \bfe_u\|_{\calT_h} 
&\le C \Big( h^{t-1} \|\bfu\|_{t, \Omega} + h^s \|\bfw\|_{s, \Omega} \Big),
\end{align*}
for $1 \le s \le k+1, \; 1 \le r \le k+1,\; 1 \le t \le k+2$. 
\end{lemma}

\begin{proof}
The first equality is trivial since both $\bfu, \bfu_h$ are divergence free over $\Omega$. For the second inequality, 
we begin by estimating $T_1, T_2$ on the right hand side of the identity in Lemma \ref{lemma_energy_id}. We start with $T_2$, 
\begin{align*}
T_2 &= \langle \tau_t \bfP_{\bfM}(\bfu^t - (\bfPi^{k+1}_{\bf BDM} \bfu)^t)  \, , \, (\bfe_u^t - \bfe_{\widehat{u}^t}) \rangle_{\partial \calT_h} = \langle \tau_t \bfP_{\bfM}(\bfu^t - (\bfPi^{k+1}_{\bf BDM} \bfu)^t)  \, , \, (\bfP_{\bfM} \bfe_u^t - \bfe_{\widehat{u}^t}) \rangle_{\partial \calT_h} \\
& = \langle \tau_t (\bfu- \bfPi^{k+1}_{\bf BDM} \bfu)^t  \, , \, (\bfP_{\bfM} \bfe_u^t - \bfe_{\widehat{u}^t}) \rangle_{\partial \calT_h}
 \le  \tau_t \|\bfu- \bfPi^{k+1}_{\bf BDM} \bfu\|_{\partial \calT_h} \|\bfP_{\bfM} \bfe_u^t - \bfe_{\widehat{u}^t}\|_{\partial \calT_h}  \\
 & \le C \tau_t^{\frac12} h^{t-\frac12} \|\bfu\|_{t, \Omega} \, \tau_t^{\frac12}\|\bfP_{\bfM} \bfe_u^t - \bfe_{\widehat{u}^t}\|_{\partial \calT_h},
\end{align*}
for $1 \le t \le k+2$. The last step we applied the approximation property of the projection $\bfPi^{k+1}_{\bf BDM}$ \eqref{classic_2}.

To bound $T_1$, we first rewrite the term as:
\[
T_1 = \langle \bfw - \bfPi_W {\bfw} \, , \, ( \bfP_{\bfM}  \bfe_u^t - \bfe_{\widehat{u}^t}) \times \bfn \rangle_{\partial \calT_h} + \langle \bfw - \bfPi_W {\bfw} \, , \, ( \bfe_u^t - \bfP_{\bfM}  \bfe_u^t) \times \bfn \rangle_{\partial \calT_h} := T_{11} + T_{12}.
\]

For $T_{11}$ we directly apply Cauchy-Schwarz inequality,
\[
T_{11} \le \tau_t^{\frac12} \|\bfP_{\bfM}  \bfe_u^t - \bfe_{\widehat{u}^t}\|_{\partial \calT_h} \; \tau_t^{-\frac12}  \| \bfw - \bfPi_W {\bfw} \|_{\partial \calT_h} \le C \tau_t^{-\frac12} h^{s-\frac12} \|\bfw\|_{s, \Omega} \; \tau_t^{\frac12} \|\bfP_{\bfM}  \bfe_u^t - \bfe_{\widehat{u}^t}\|_{\partial \calT_h},
\]
the last inequality is due to \eqref{classic_4}.

For $T_{12}$ we have
\begin{align*}
T_{12} &= - \langle (\bfw - \bfPi_W {\bfw}) \times \bfn \, , \, ( \bfe_u^t - \bfP_{\bfM}  \bfe_u^t)  \rangle_{\partial \calT_h} \\
& =  - \langle \bfw \times \bfn - \bfP_{\bfM} ({\bfw} \times \bfn) \, , \, ( \bfe_u^t - \bfP_{\bfM}  \bfe_u^t)  \rangle_{\partial \calT_h} \qquad  \text{(due to $(\bfPi_W \bfw \times \bfn)|_{F} \in \bfM^t(F)$)} \\
& =  - \langle \bfw \times \bfn - \bfP_{\bfM} ({\bfw} \times \bfn) \, , \,  \bfe_u^t  \rangle_{\partial \calT_h} \\
& =  - \langle \bfw \times \bfn - \bfP_{\bfM} ({\bfw} \times \bfn) \, , \,  \bfe_u^t + (\nabla \eta)^t  \rangle_{\partial \calT_h},\\
\intertext{for any $\eta \in P_{k+2}(\calT_h):=\oplus_{K \in \calT_h} P_{k+2}(K)$. Due to the orthogonal property of $\bfP_{\bfM}$ and the definition of the space $\bfM_h^t$, applying Cauchy-Schwarz inequality and inverse inequality we have,}
T_{12} & \le \| \bfw \times \bfn - \bfP_{\bfM} ({\bfw} \times \bfn)\|_{\partial \calT_h}  \|\bfe_u^t + (\nabla \eta)^t\|_{\partial \calT_h} \le \| \bfw \times \bfn - \bfP_{\bfM} ({\bfw} \times \bfn) \|_{\partial \calT_h}  \|\bfe_u^t + (\nabla \eta)^t\|_{\partial \calT_h} \\
           & \le C \| \bfw - \bfP_{\bfM} {\bfw}\|_{\partial \calT_h} \; h^{-\frac12} \|\bfe_u + \nabla \eta\|_{\calT_h},
\end{align*}
for any $\eta \in P_{k+2}(\calT_h)$. We notice that on each $K \in \calT_h$, the $\nabla \times$ operator is an injective mapping from $\bfP_{k+1}(K)\setminus \nabla (P_{k+2}(K))$ onto $\bfP_{k}(K)$, 
therefore, the two norms:
\[
\min_{\eta \in P_{k+2}(K)} \|\bfv + \nabla \eta\|_{K}, \quad \|\nabla \times \bfv\|_{K}
\]
are equivalent in the finite dimensional space $\bfP_{k+1}(K)$. A simple scaling argument implies that 
\begin{equation}\label{norm_equiv}
\min_{\eta \in P_{k+2}(K)} \|\bfv + \nabla \eta\|_{K} \le C h  \|\nabla \times \bfv\|_{K}.
\end{equation}
Owing to the above inequality and \eqref{classic_4} we have
\begin{align*}
T_{12} &\le C \| \bfw - \bfPi_{W} {\bfw}\|_{\partial \calT_h} \; h^{-\frac12} \min_{\eta \in P_{k+2}(\calT_h)} \|\bfe_u + \nabla \eta\|_{\calT_h} \\
& \le C h^s \|\bfw\|_{s, \Omega} \|\nabla \times \bfe_u\|_{\calT_h},
\end{align*}
for $1 \le s \le k+1$. 

To conclude the estimate for $\|\bfe_{w}\|_{\Omega}$, we need to bound $\|\nabla \times \bfe_{u}\|_{\calT_h}$. To this end, taking $\bfr = \nabla \times \bfe_{u}$ in the error equation \eqref{error_1} and integrating by parts, we have
\[
(\bfe_w, \nabla \times \bfe_u)_{\calT_h} - \|\nabla \times \bfe_u\|_{\calT_h}^2 - \langle (\bfe_u - \bfe_{\widehat{u}^t}) \times \bfn \, , \, \nabla \times \bfe_u \rangle_{\partial \calT_h} = 0. 
\]
By the fact that $\bfn\times(  \nabla \times \bfe_u)|_{F} \in \bfM^t(F)$ for all $F \in \partial \calT_h$, we can rewrite the above identity as
\begin{align*}
\|\nabla \times \bfe_u\|_{\calT_h}^2 &= (\bfe_w, \nabla \times \bfe_u)_{\calT_h} - \langle (\bfP_{\bfM} \bfe^t_u - \bfe_{\widehat{u}^t}) \times \bfn \, , \, \nabla \times \bfe_u \rangle_{\partial \calT_h} \\
&\le C \|\nabla \times \bfe_u\|_{\calT_h} ( \|\bfe_w\|_{\Omega} + h^{-\frac12} \|\bfP_{\bfM} \bfe^t_u - \bfe_{\widehat{u}^t}\|_{\partial \calT_h}).
\end{align*}
If we combine the estimates for $T_1, T_2$ and $\|\nabla \times \bfe_u\|_{\calT_h}$ and apply the Young's inequality, we have
\begin{align*}
\|\bfe_w\|_{\calT_h} &+ \tau^{\frac12}_t \|\bfP_{\bfM} \bfe_u^t - \bfe_{\widehat{u}^t}\|_{\partial \calT_h}
\le  C \Big( \tau_t^{\frac12}  h^{t-\frac12} \|\bfu\|_{t, \Omega} +(1 + \tau_t^{-\frac12} h^{-\frac12}) h^s \|\bfw\|_{s, \Omega} \Big),
\end{align*}
for $1 \le s \le k+1, \; 1 \le r \le k+1,\;1 \le t \le k+2$. We can see that if we take $\tau_t = \mathcal{O}(\frac{1}{h})$, we obtain the estimate stated in the lemma.
\end{proof}

Next we apply a duality argument to show the superconvergence of $\|\bfe_u\|_{\Omega}$. We first present the following identity that we are going to use in the analysis.

\begin{lemma}\label{dual_identity}
Let $(\bfphi, \bftheta, \sigma)$ be the solution of the adjoint problem \eqref{adj_eq}, then we have
\begin{align*}
\|\bfe_u\|^2_{\Omega}  = & \langle (\bfe_u^t - \bfe_{\widehat{u}^t})\times \bfn \, , \, \bftheta - \bfPi_W \bftheta \rangle_{\partial \calT_h} - \langle e_p -  e_{\widehat{p}} \, , \, (\bfphi - \bfPi^{k+1}_{\bf RT} \bfphi) \cdot \bfn \rangle_{\partial \calT_h} \\
&+\langle \bfw - \widehat{\bfw}_h -\bfe_w \, , \, (\bfphi - \bfPi^{k+1}_{\bf RT} \bfphi)\times \bfn \rangle_{\partial \calT_h}  + \langle \bfu \cdot \bfn - {\widehat{\bfu}_h} \cdot \bfn -\bfe_u \cdot \bfn \, , \, \sigma - \Pi_Q \sigma \rangle_{\partial \calT_h} \\
:= &  T_1 - T_2 + T_3 + T_4.
\end{align*}
Here $\bfPi^{k+1}_{\bf RT}$ denotes the $\bf{RT}$ projection onto the local space $\bf{RT}_{k+1}(K)$.
\end{lemma}

\begin{proof}
Notice that by the dual equation \eqref{adj_3} we know that $\bfphi$ is divergence free. As a consequence, we have $\bfPi^{k+1}_{\bf RT} \bfphi \in \bfV_h$. 
By the adjoint equation \eqref{adj_eq} we have
\begin{align*}
\|\bfe_u\|^2_{\Omega}  &=  (\bfe_u, \nabla \times \bftheta - \nabla \sigma)_{\calT_h} - (\bfe_w, \bftheta - \nabla \times \bfphi)_{\calT_h} - (e_p, \nabla \cdot \bfphi)_{\calT_h} \\
&= (\bfe_u, \nabla \times \bftheta)_{\calT_h} - (\bfe_w, \bftheta)_{\calT_h} \\
&\quad + (\bfe_w, \nabla \times \bfphi)_{\calT_h} - (e_p, \nabla \cdot \bfphi)_{\calT_h} \\
& \quad - (\bfe_u, \nabla \sigma)_{\calT_h} \\
&= (\bfe_u, \nabla \times \bfPi_W \bftheta)_{\calT_h} - (\bfe_w, \bfPi_W \bftheta)_{\calT_h} +  (\bfe_u, \nabla \times (\bftheta - \bfPi_W \bftheta))_{\calT_h} - (\bfe_w, \bftheta - \bfPi_W \bftheta)_{\calT_h} \\
&\quad + (\bfe_w, \nabla \times \bfPi^{k+1}_{\bf RT} \bfphi)_{\calT_h} - (e_p, \nabla \cdot \bfPi^{k+1}_{\bf RT} \bfphi)_{\calT_h} + (\bfe_w, \nabla \times(\bfphi - \bfPi^{k+1}_{\bf RT} \bfphi))_{\calT_h} - (e_p, \nabla \cdot(\bfphi - \bfPi^{k+1}_{\bf RT} \bfphi))_{\calT_h} \\
&\quad - (\bfe_u, \nabla \Pi_Q \sigma)_{\calT_h} - (\bfe_u, \nabla ( \sigma - \Pi_Q \sigma))_{\calT_h}. \\
\intertext{By integrating by parts and the orthogonal property of the projections, we have}
\|\bfe_u\|^2_{\Omega} &= (\bfe_u, \nabla \times \bfPi_W \bftheta)_{\calT_h} - (\bfe_w, \bfPi_W \bftheta)_{\calT_h} +  \langle \bfe^t_u \times \bfn \, , \,  \bftheta - \bfPi_W \bftheta \rangle_{\partial \calT_h} \\
&\quad+ (\bfe_w, \nabla \times \bfPi^{k+1}_{\bf RT} \bfphi)_{\calT_h} - (e_p, \nabla \cdot \bfPi^{k+1}_{\bf RT} \bfphi)_{\calT_h} - \langle \bfe_w \, , \, (\bfphi - \bfPi^{k+1}_{\bf RT} \bfphi) \times \bfn \rangle_{\partial \calT_h} -  \langle e_p\, , \, (\bfphi - \bfPi^{k+1}_{\bf RT} \bfphi) \cdot \bfn \rangle_{\partial \calT_h} \\
& \quad - (\bfe_u, \nabla \Pi_Q \sigma)_{\calT_h} - \langle \bfe_u \cdot \bfn\, , \,  \sigma - \Pi_Q \sigma \rangle_{\partial \calT_h} .\\
\intertext{Taking $(\bfr, \bfv, q) = (\bfPi_W \bftheta, \bfPi^{k+1}_{\bf RT} \bfphi, \Pi_Q \sigma)$ in the error equations \eqref{error_1}-\eqref{error_3}, inserting these equations into the above identity we have}
\|\bfe_u\|^2_{\Omega} &=  \langle \bfe_{\widehat{u}^t} \times \bfn \, , \, \bfPi_W \bftheta \rangle_{\partial \calT_h}  +  \langle \bfe^t_u \times \bfn \, , \,  \bftheta - \bfPi_W \bftheta \rangle_{\partial \calT_h} \\
& - \langle \bfw - \widehat{\bfw}_h \, , \, \bfPi^{k+1}_{\bf RT} \bfphi \times \bfn \rangle_{\partial \calT_h} -  \langle e_{\widehat{p}}\, , \,  \bfPi^{k+1}_{\bf RT} \bfphi \cdot \bfn \rangle_{\partial \calT_h} \\
&- \langle \bfe_w \, , \, (\bfphi - \bfPi^{k+1}_{\bf RT} \bfphi) \times \bfn \rangle_{\partial \calT_h} -  \langle e_p\, , \, (\bfphi - \bfPi^{k+1}_{\bf RT} \bfphi) \cdot \bfn \rangle_{\partial \calT_h} \\
&   - \langle \bfu \cdot \bfn - \widehat{\bfu}^n_h \cdot \bfn \, , \, \Pi_Q \sigma \rangle_{\partial \calT_h} - \langle \bfe_u \cdot \bfn\, , \,  \sigma - \Pi_Q \sigma \rangle_{\partial \calT_h}.
\end{align*}
Notice by the regularity assumption of the adjoint problem, and \eqref{original_4}, \eqref{original_5}, 
\eqref{HDG_form_4}, \eqref{HDG_form_5}, we have 
\[
- \langle \bfe_{\widehat{u}^t} \times \bfn \, , \, \bftheta \rangle_{\partial \calT_h}=0,\;  \langle \bfw - \widehat{\bfw}_h \, , \, \bfphi \times \bfn \rangle_{\partial \calT_h} 
 = 0,\; \langle e_{\widehat{p}}\, , \, \bfphi \cdot \bfn \rangle_{\partial \calT_h} =0,\; \langle \bfu \cdot \bfn - \widehat{\bfu}^n_h \cdot \bfn \, , \,  \sigma \rangle_{\partial \calT_h} = 0.
\]
Inserting the above zero terms in the last expression of $\|\bfe_u\|^2_{\Omega}$ we obtain the desired identity in the lemma. This completes the proof.
\end{proof}

Finally, we are ready to present the estimate for $\|\bfe_u\|_{\Omega}$.
\begin{lemma}\label{estimate_eu}
Under the same assumption as Lemma \ref{estimate_ew}, in addition if the regularity assumption \eqref{regularity} holds for the adjoint problem, then we have 
\[
\|\bfe_u\|_{\Omega} \le C \Big( h^{t+\alpha-1} \|\bfu\|_{t, \Omega} + h^{s+\alpha} \|\bfw\|_{s, \Omega} \Big),
\]
for $1 \le s \le k+1, \, 1 \le r \le k+1,\, 1 \le t \le k+2$ and $0<\alpha\leq 1$ is defined in \eqref{regularity}.
\end{lemma}

\begin{proof}
We will estimate $T_1$-$T_4$ in the above Lemma separately. First let us show that in fact $T_2, T_4$ vanish.  Notice that on each $F \in \partial \mathcal{T}_h$, $e_p - e_{\widehat{p}} \in P_{k+1}(F)$, by the property of the {\bf RT} projection we have
\begin{align*}
T_2 = \langle e_p -  e_{\widehat{p}} \, , \, (\bfphi - \bfPi^{k+1}_{\bf RT} \bfphi) \cdot \bfn \rangle_{\partial \calT_h} = 0. 
\end{align*}
For the last term $T_4$, we have
\begin{align*}
T_4 & = \langle \bfu \cdot \bfn - {\widehat{\bfu}_h} \cdot \bfn -\bfe_u \cdot \bfn \, , \, \sigma - \Pi_Q \sigma \rangle_{\partial \calT_h} \\
& =  \langle (\bfu - \bfPi^{k+1}_{\bf BDM} \boldsymbol{u})\cdot \bfn - ({\widehat{\bfu}_h} \cdot \bfn -\boldsymbol{u}_h \cdot \bfn) \, , \, \sigma - \Pi_Q \sigma \rangle_{\partial \calT_h}\\
& =  \langle (\bfu - \bfPi^{k+1}_{\bf BDM} \boldsymbol{u})\cdot \bfn\, , \, \sigma - \Pi_Q \sigma \rangle_{\partial \calT_h} \\
& =  \langle (\bfu - \bfPi^{k+1}_{\bf BDM} \boldsymbol{u})\cdot \bfn\, , \, \sigma \rangle_{\partial \calT_h} \\
& =  \langle (\bfu - \bfPi^{k+1}_{\bf BDM} \boldsymbol{u})\cdot \bfn\, , \, \sigma \rangle_{\partial \Omega} = 0. 
\end{align*}
For $T_1$, we rewrite it into two parts:
\[
T_1 = \langle (\bfP_{\bfM} \bfe_u^t - \bfe_{\widehat{u}^t})\times \bfn \, , \, \bftheta - \bfPi_W \bftheta \rangle_{\partial \calT_h} + \langle (\bfe_u^t - \bfP_{\bfM} \bfe_u^t )\times \bfn \, , \, \bftheta - \bfPi_W \bftheta \rangle_{\partial \calT_h} := T_{11} + T_{12}.
\]
Applying Cauchy-Schwarz inequality on $T_{11}$ and then \eqref{classic_4}, we have
\begin{align*}
T_{11} &\le \|(\bfP_{\bfM} \bfe_u^t - \bfe_{\widehat{u}^t})\times \bfn\|_{\partial \calT_h} \|\bftheta - \bfPi_W \bftheta\|_{\partial \calT_h} \\
& \le \tau_t^{\frac12} \|(\bfP_{\bfM} \bfe_u^t - \bfe_{\widehat{u}^t})\times \bfn\|_{\partial \calT_h} \; \tau_t^{-\frac12} h^{\alpha-\frac12} \|\bftheta\|_{\alpha, \Omega} \\
& \le C \tau_t^{-\frac12} h^{\alpha-\frac12} \|\bfe_u\|_{\Omega} \;  \tau_t^{\frac12} \|(\bfP_{\bfM} \bfe_u^t - \bfe_{\widehat{u}^t})\times \bfn\|_{\partial \calT_h},
\end{align*}
the last step is due to the regularity condition \eqref{regularity}. For $T_{12}$, we will apply a similar argument as the estimate for $T_{12}$ in Lemma \ref{estimate_ew}. Namely, for any $\eta \in P_{k+2}(\calT_h)$, by \eqref{classic_4} and \eqref{regularity}, we have
\begin{align*}
T_{12} &=   \langle ((\bfe_u^t + (\nabla \eta)^t) -  \bfP_{\bfM} (\bfe_u^t + (\nabla \eta)^t) )\times \bfn \, , \, \bftheta - \bfPi_W \bftheta \rangle_{\partial \calT_h} \\
&\le \|((\bfe_u^t + (\nabla \eta)^t) -  \bfP_{\bfM} (\bfe_u^t + (\nabla \eta)^t) )\|_{\partial \calT_h} \|\bftheta - \bfPi_W \bftheta\|_{\partial \calT_h} \\
&\le \|\bfe_u^t + (\nabla \eta)^t \|_{\partial \calT_h} \|\bftheta - \bfPi_W \bftheta\|_{\partial \calT_h} \\
& \le C h^{-\frac12} \|\bfe_u + \nabla \eta\|_{\calT_h} \|\bftheta - \bfPi_W \bftheta\|_{\partial \calT_h}.
\end{align*}
By the similar technique in (\ref{norm_equiv}), we have
\begin{align*}
T_{12} &\le C h^{-\frac12} \min_{\eta \in P_{k+2}(\calT_h)}  \|\bfe_u + \nabla \eta\|_{\calT_h} \|\bftheta - \bfPi_W \bftheta\|_{\partial \calT_h} \\
&\le C h^{\frac12} \|\nabla \times \bfe_u\|_{\calT_h}  \|\bftheta - \bfPi_W \bftheta\|_{\partial \calT_h} \\
& \le C h^{\frac12} \|\nabla \times \bfe_u\|_{\calT_h} h^{\alpha-\frac12} \|\bftheta\|_{\alpha, \Omega} \\
&\le  C h^{\alpha} \|\nabla \times \bfe_u\|_{\calT_h}  \|\bfe_u\|_{\Omega}.
\end{align*}
For $T_3$, we can use the expressions in the proof of Lemma \ref{lemma_energy_id} below \eqref{energy_1}. 
\begin{align*}
T_3 & = \langle \bfw -\bfPi_W \bfw + \tau_t (\bfP_{\bfM} \bfe_u^t - \bfe_{\widehat{\bfu}^t}) \times \bfn + \tau_t \bfP_{\bfM}(\bfu^t - (\bfPi^{k+1}_{\bf RT} \bfu)^t) \times \bfn \, , \, (\bfphi - \bfPi^{k+1}_{\bf RT} \bfphi)\times \bfn \rangle_{\partial \calT_h} \\
& \le \|\bfphi - \bfPi^{k+1}_{\bf RT} \bfphi\|_{\partial \calT_h} (\| \bfw -\bfPi_W \bfw \|_{\partial \calT_h} + \tau_t \|\bfP_{\bfM} \bfe_u^t - \bfe_{\widehat{\bfu}^t}\|_{\partial \calT_h} + \tau_t \|\bfP_{\bfM}(\bfu^t - (\bfPi^{k+1}_{\bf RT} \bfu)^t)\|_{\partial \calT_h}) \\
& \le \|\bfphi - \bfPi^{k+1}_{\bf RT} \bfphi\|_{\partial \calT_h} (\| \bfw -\bfPi_W \bfw \|_{\partial \calT_h} + \tau_t \|\bfP_{\bfM} \bfe_u^t - \bfe_{\widehat{\bfu}^t}\|_{\partial \calT_h} + \tau_t \|\bfu - \bfPi^{k+1}_{\bf RT} \bfu\|_{\partial \calT_h}). \\
\intertext{By \eqref{classic_2}, \eqref{classic_4} and then \eqref{regularity}, we obtain}
T_3 & \le C \tau_t^{\frac12} h^{\alpha+\frac12} \|\bfe_u\|_{\Omega} \; \tau_t^{\frac12} \|\bfP_{\bfM} \bfe_u^t - \bfe_{\widehat{\bfu}^t}\|_{\partial \calT_h} + C \|\bfe_u\|_{\Omega} (h^{s+\alpha} \|\bfw\|_{s, \Omega} + \tau_t h^{t+\alpha} \|\bfu\|_{\Omega}),
\end{align*}
for all $1 \le s \le k+1, \, 1 \le t \le k+2$. The proof is complete if we combine the estimates for $T_1$-$T_4$ and taking $\tau_t = \mathcal{O}(\frac{1}{h})$.
\end{proof}

Finally, by triangle inequality and \eqref{classic_1}, \eqref{classic_3}, Theorem \ref{main_mod} is a direct consequence of Lemmas \ref{estimate_ew}, \ref{estimate_eu}. 

\subsection{Error estimates for $\text{HDG}_g$ on general polyhedral meshes}

%
%
The proof for this case follows the main path as the above case. Therefore, we only briefly sketch the proof and illustrate all the different steps from the above proofs.

First of all, for the $\text{HDG}_g$ on general polyhedral meshes, we have an additional stabilization parameter $\tau_n > 0$, and the generic element $K$ could be polyhedral. As a consequence, the above proof for $\text{HDG}_s$ needs to be modified since the projections $\bfPi^{k+1}_{\bf BDM}, \bfPi^{k+1}_{\bf RT}$ are tailored for simplex elements. In this case in the analysis we will replace both mixed type projections with the standard $L^2$-projection, denoted by $\bfPi_V$, onto the space $\bfV_h$ and the projection error $\bfe_u := \bfPi_V \bfu - \bfu_h$. With this changes, we can easily verify that the error equations \eqref{error_eqs} in Lemma \ref{error_eq} remains the same form. In addition, the $L^2$-projection has similar approximation properties as the {\bf BDM} projection:
\begin{subequations}
\begin{align}
\label{classic_1t}
\|\bfu - \bfPi_V \bfu \|_{\calT_h} & \le C h^{t-\delta} \|\bfu\|_{t, \Omega}, && 0 \le t \le k+2,\, \delta = 0,1,  \\
\label{classic_2t}
\|\bfu - \bfPi_V \bfu \|_{\partial \calT_h} & \le C h^{t- \frac12} \|\bfu\|_{t, \Omega},  && 1 \le t \le k+2, 
\end{align}
\end{subequations}

By a similar energy argument, we have the following energy identity result:

\begin{lemma}\label{energy_id_g}
We have
\begin{align*}
&\quad  \|\bfe_w\|^2_{\Omega} + \tau_t \|\bfP_{\bfM} \bfe_u^t - \bfe_{\widehat{u}^t}\|^2_{\partial \calT_h} + \tau_n \|e_p - e_{\widehat{p}}\|^2_{\partial \calT_h}  \\
&=  - \langle \bfw - \bfPi_W {\bfw} \, , \, (\bfe_u^t - \bfe_{\widehat{u}^t}) \times \bfn \rangle_{\partial \calT_h}   - \langle \tau_t \bfP_{\bfM}(\bfu^t - (\bfPi_V \bfu)^t) \times \bfn \, , \, (\bfe_u^t - \bfe_{\widehat{u}^t}) \times \bfn \rangle_{\partial \calT_h} \\
& \quad - \langle (\bfu - \bfPi_V \bfu) \cdot \bfn \, , \, e_p - e_{\widehat{p}} \rangle_{\partial \calT_h} - \langle \tau_n (P_M p - \Pi_Q p)  \, , \, e_p - e_{\widehat{p}} \rangle_{\partial \calT_h} \\
&:= - T_1 - T_2 - T_3 - T_4.
\end{align*}
\end{lemma}

Notice that in the above identity $T_1$ and $T_2$ are the same as before with the modified projection on $\bfu$, nevertheless, we can bound these two terms with the same argument as in Lemma \ref{estimate_ew}.

For $T_3$ we directly apply Cauchy-Schwarz inequality and \eqref{classic_2t} to obtain: 
\[
T_3 \le \|\bfu- \bfPi_V \bfu\|_{\partial \calT_h} \|e_p - e_{\widehat{p}}\|_{\partial \calT_h} \le C \tau_n^{-\frac12} h^{t- \frac12} \|\bfu\|_{t, \Omega} \; \tau_n^{\frac12} \|e_p - e_{\widehat{p}}\|_{\partial \calT_h},
\]
for $1 \le t \le k+2$.

Similarly, for $T_4$ by the Cauchy-Schwarz inequality, triangle inequality and \eqref{classic_6}-\eqref{classic_7}, we have
\begin{align*}
T_4  \le \tau_n^{\frac12} \|P_M p - \Pi_Q p\|_{\partial \calT_h} \; \tau_n^{\frac12} \|e_p - e_{\widehat{p}}\|_{\partial \calT_h} 
\le C \tau^{\frac12}_n h^{r - \frac12} \|p\|_{r, \Omega} \; \tau_n^{\frac12} \|e_p - e_{\widehat{p}}\|_{\partial \calT_h}.
\end{align*}

Combining these two estimates together with the estimates for $T_1, T_2$ in Lemma \ref{estimate_ew}, we obtain the following estimates:

\begin{lemma}\label{estimate_ew_g}
If the regularity of the exact solution $(\bfu, \bfw, p)$ of (\ref{original}) holds as assumed in Theorem \ref{main_result}, the stabilization parameters $\tau_t = \mathcal{O}(\frac{1}{h}), \tau_n = \mathcal{O}(h)$ on each $F \in \partial \calT_h$, we have
\begin{align*}
&\quad \|\bfe_w\|_{\Omega} + \tau^{\frac12}_t \|\bfP_{\bfM} \bfe_u^t - \bfe_{\widehat{u}^t}\|_{\partial \calT_h} + \tau^{\frac12}_n \|e_p - e_{\widehat{p}}\|_{\partial \calT_h} + \|\nabla \times \bfe_u\|_{\calT_h} + \|\nabla \cdot \bfe_u\|_{\calT_h} \\
&\le C \Big( h^{t-1} \|\bfu\|_{t, \Omega} + h^{r}\|p\|_{r, \Omega} + h^s \|\bfw\|_{s, \Omega} \Big),
\end{align*}
for $1 \le s \le k+1, \; 1 \le r \le k+1,\; 1 \le t \le k+2$. 
\end{lemma}

As we mentioned at the beginning of the subsection, in the duality argument we need to replace the {\bf RT} projection with the standard $L^2$-projection. Nevertheless, the dual identity in Lemma \ref{dual_identity} is still valid with the projection modified:

\begin{lemma}\label{dual_identity_g}
Let $(\bfphi, \bftheta, \sigma)$ be the solution of the adjoint problem \eqref{adj_eq}, then we have
\begin{align*}
\|\bfe_u\|^2_{\Omega}  = & \langle (\bfe_u^t - \bfe_{\widehat{u}^t})\times \bfn \, , \, \bftheta - \bfPi_W \bftheta \rangle_{\partial \calT_h} - \langle e_p -  e_{\widehat{p}} \, , \, (\bfphi - \bfPi_V \bfphi) \cdot \bfn \rangle_{\partial \calT_h} \\
&+\langle \bfw - \widehat{\bfw}_h -\bfe_w \, , \, (\bfphi - \bfPi_V \bfphi)\times \bfn \rangle_{\partial \calT_h}  + \langle \bfu \cdot \bfn - {\widehat{\bfu}_h} \cdot \bfn -\bfe_u \cdot \bfn \, , \, \sigma - \Pi_Q \sigma \rangle_{\partial \calT_h} \\
:= &  T_1 - T_2 + T_3 + T_4.
\end{align*}
\end{lemma}

Indeed, the same argument can be applied to bound $T_1, T_3$ as in Lemma \ref{estimate_eu} but $T_2, T_4$ no longer vanish. We can bound them as follows:
For $T_{2}$, we simply apply Cauchy-Schwarz inequality and then \eqref{classic_2t}, \eqref{regularity} to get
\begin{align*}
T_2 &\le \tau_n^{\frac12} \|e_p - e_{\widehat{p}}\|_{\partial \calT_h} \; \tau_n^{-\frac12} \|\bfphi - \bfPi_V \bfphi\|_{\partial \calT_h} \le \tau_n^{-\frac12} h^{\alpha+\frac12} \|\bfphi\|_{1+\alpha, \Omega} \; \tau_n^{\frac12} \|e_p - e_{\widehat{p}}\|_{\partial \calT_h} \\
& \le C  \tau_n^{-\frac12} h^{\alpha+\frac12} \|\bfe_u\|_{\Omega} \; \tau_n^{\frac12} \|e_p - e_{\widehat{p}}\|_{\partial \calT_h}.
\end{align*}
For $T_4$, it can be bounded in a similar way as for $T_3$ in the proof of Lemma \ref{estimate_eu} which gives us:
\[
T_4 \le C \tau_n^{\frac12} h^{\alpha-\frac12} \|\bfe_u\|_{\Omega} \, \tau_n^{\frac12} \|e_p - e_{\widehat{p}}\|_{\partial \calT_h} + C \|\bfe_u\|_{\Omega} (h^{t+\alpha-1} \|\bfu\|_{t, \Omega} + \tau_n h^{r+\alpha-1} \|p\|_{r, \Omega}),
\]
for $1 \le r \le k+1,\;1 \le t \le k+2$.

If we combine the above two estimates together with the bounds for $T_1, T_3$ in Lemma \ref{estimate_eu} we obtain the following result for $\bfe_u$:
\begin{lemma}\label{estimate_eu_g}
Under the same assumption as Lemma \ref{estimate_ew_g}, in addition if the regularity assumption \eqref{regularity} holds for the adjoint problem, then we have 
\[
\|\bfe_u\|_{\Omega} \le C \Big( h^{t+\alpha-1} \|\bfu\|_{t, \Omega} + h^{r+\alpha} \|p\|_{r,\Omega} + h^{s+\alpha} \|\bfw\|_{s, \Omega} \Big),
\]
for $1 \le s \le k+1, \, 1 \le r \le k+1,\, 1 \le t \le k+2$ and $0<\alpha\leq 1$ is defined in \eqref{regularity}.
\end{lemma}

Finally, Theorem \ref{main_result} is a direct consequence of Lemmas \ref{estimate_ew_g}, \ref{estimate_eu_g} with a triangle inequality and \eqref{classic_1t}, \eqref{classic_3}. 
\begin{remark}
Besides we get the error estimate for $\tau^{\frac12}_n \|e_p - e_{\widehat{p}}\|_{\partial \calT_h}$ in Lemma \ref{estimate_ew}, 
we can also derive the error estimate for $\|\nabla e_p\|_\Omega$. Actually, this can be obtained by taking $\bfv = \nabla e_p$ in the error equation (\ref{error_2}) and applying integration by parts, the Cauchy- Schwarz inequality, trace inequality and the estimates in Lemma \ref{estimate_ew} and Theorem \ref{main_result}. Then the error estimate for $\|\nabla(p-p_h) \|_\Omega$ can be further deduced by the triangular inequality.
\end{remark}

\section{Numerical results}

In this section we provide numerical experiments validating the error estimates obtained in this work. In the following simulations we consider tetrahedral meshes and $\tau_t=h^{-1}$. The implementation is based on the work developed by \cite{Fu2015}.

\begin{example}\label{example1}
(Smooth case) We test the HDG$_s$ method on a steady state Maxwell problem on a unit cube $\Omega = [0,1]^3$, where the source term $\boldsymbol{f}$ and the boundary conditions are chosen such that the exact solutions are
\begin{align*}
\bfu(x,y,z) &= \big(\sin(\pi y)\sin(\pi z),\sin(\pi x)\sin(\pi z),\sin(\pi x)\sin(\pi y)\big)^T, \\
p(x,y,z) &= \sin(2\pi x)\sin(2 \pi y) \sin(2 \pi z).
\end{align*}

In Table \ref{table1} we display the history of convergence of the $L^2$-error of $\bfu$ and $\bfw$ for $k=0$, $1$, $2$ and $3$. We observe that optimal rate of convergence are obtained, i.e., order of $h^{k+1}$ for $\bfw$ and order of $h^{k+2}$ for $\bfu$.

\begin{table}[ht!]\renewcommand{\arraystretch}{1.2}
\begin{tabular}{c|c||c c|c c}
$k$ & $h$ & $\|\boldsymbol{u}-\boldsymbol{u}_h\|_{\Omega}$ & $\text{order}$ &$\|\boldsymbol{w}-\boldsymbol{w}_h\|_{\Omega}$ & $\text{order}$\\\hline\hline
 & $4.39E-001$ & $6.81E-001$ & $-$ & $2.46E+000$ & $-$ \\ 
 & $3.29E-001$ & $3.66E-001$ & $2.16$ & $1.86E+000$ & $0.96$ \\ 
$0$ & $2.63E-001$ & $2.28E-001$ & $2.11$ & $1.50E+000$ & $0.98$ \\ 
 & $2.19E-001$ & $1.56E-001$ & $2.08$ & $1.25E+000$ & $0.99$ \\ 
 & $1.88E-001$ & $1.14E-001$ & $2.06$ & $1.07E+000$ & $0.99$ \\ 
 & $1.65E-001$ & $8.63E-002$ & $2.05$ & $9.40E-001$ & $0.99$ \\ 
\hline \hline 
& $4.39E-001$ & $1.12E-001$ & $-$ & $3.10E-001$ & $-$ \\ 
 & $3.29E-001$ & $4.80E-002$ & $2.95$ & $1.80E-001$ & $1.88$ \\ 
$1$  & $2.63E-001$ & $2.49E-002$ & $2.95$ & $1.17E-001$ & $1.92$ \\ 
 & $2.19E-001$ & $1.45E-002$ & $2.95$ & $8.25E-002$ & $1.94$ \\ 
 & $1.88E-001$ & $9.21E-003$ & $2.96$ & $6.11E-002$ & $1.95$ \\ 
 & $1.65E-001$ & $6.20E-003$ & $2.96$ & $4.71E-002$ & $1.96$ \\ 
\hline \hline 
& $4.39E-001$ & $1.24E-002$ & $-$ & $6.22E-002$ & $-$ \\ 
 & $3.29E-001$ & $3.67E-003$ & $4.22$ & $2.67E-002$ & $2.94$ \\ 
$2$  & $2.63E-001$ & $1.45E-003$ & $4.17$ & $1.38E-002$ & $2.96$ \\ 
 & $2.19E-001$ & $6.80E-004$ & $4.14$ & $8.03E-003$ & $2.97$ \\ 
 & $1.88E-001$ & $3.61E-004$ & $4.11$ & $5.07E-003$ & $2.98$ \\ 
 & $1.65E-001$ & $2.09E-004$ & $4.09$ & $3.41E-003$ & $2.98$ \\ 
\hline \hline 
& $4.39E-001$ & $1.40E-003$ & $-$ & $6.93E-003$ & $-$ \\ 
 & $3.29E-001$ & $3.31E-004$ & $5.01$ & $2.23E-003$ & $3.94$ \\ 
$3$  & $2.63E-001$ & $1.09E-004$ & $4.99$ & $9.22E-004$ & $3.96$ \\ 
 & $2.19E-001$ & $4.37E-005$ & $4.99$ & $4.47E-004$ & $3.97$ \\ 
 & $1.88E-001$ & $2.03E-005$ & $4.99$ & $2.42E-004$ & $3.98$ \\ 
 & $1.65E-001$ & $1.04E-005$ & $4.99$ & $1.42E-004$ & $3.98$ \\ 
\end{tabular}
\caption{History of convergence of Example \ref{example1}.}\label{table1}
\end{table}
\end{example}

\begin{example}\label{example2}
(Solution with low regularity) This example is based on the numerical experiment presented in \cite{Houston05}.  
We assume $\Omega$ is a L-shaped domain as $\Omega = [-1,1]\times [-1,1]\times[0,1]\setminus ([0,1]
\times [-1,0]\times [0,1])$. Let the source term $\boldsymbol{f}$ and the boundary conditions be chosen such that 
the exact solutions are
\begin{align*}
\bfu &= \bigg(\frac{\partial S}{\partial x}, \frac{\partial S}{\partial y}, 0\bigg)^T,\\
p &= 0,
\end{align*}
where the function $S$ is given in by $\displaystyle S(r,\theta) = r^{\frac{4}{3}}\sin\bigg(\frac{4 \theta}{3}\bigg)$. 
The solution is given in terms of cylindrical coordinates. 
The exact solution of $\bfu$ has a singularity located at $z$-axis, and $\bfu \in [H^{\frac{4}{3}-\epsilon}
(\Omega)]^3,\epsilon>0$; hence, we only display the results corresponding to polynomials of degree $k=0$ and $k=1$.  

In Table \ref{table2} the history of convergence of the HDG$_s$ method is displayed. Here we observe 
that the convergence rates deteriorate as Theorem \ref{main_mod} predicts. In fact,  $\bfu_h$ converges to $\bfu$ with 
order $h^{4/3-\epsilon}$ and $\boldsymbol{w}_h$ converges to $\boldsymbol{w}$ with order $h^{1/3-\epsilon}$.

\begin{table}[ht!]
\begin{tabular}{c|c||c c|c c}
$k$ & $h$ & $\|\boldsymbol{u}-\boldsymbol{u}_h\|_{\Omega}$ & $\text{order}$ &$\|\boldsymbol{w}-\boldsymbol{w}_h\|_{\Omega}$ 
& $\text{order}$\\\hline\hline
& $5.00E-01$ & $1.53E-01$ & $-$ & $4.83E-02$ & $-$ \\ 
 & $2.50E-01$ & $6.24E-02$ & $1.30$ & $4.99E-02$ & $-$ \\ 
$0$  & $1.25E-01$ & $2.53E-02$ & $1.30$ & $4.40E-02$ & $0.18$ \\ 
 & $6.25E-02$ & $1.03E-02$ & $1.30$ & $3.64E-02$ & $0.27$ \\ 
 & $3.12E-02$ & $4.25E-03$ & $1.27$ & $2.94E-02$ & $0.31$ \\ 
\hline \hline 
& $5.00E-01$ & $6.66E-02$ & $-$ & $2.19E-02$ & $-$ \\ 
 & $2.50E-01$ & $2.61E-02$ & $1.35$ & $2.12E-02$ & $0.04$ \\ 
$1$  & $1.25E-01$ & $1.03E-02$ & $1.34$ & $1.78E-02$ & $0.26$ \\ 
 & $6.25E-02$ & $4.11E-03$ & $1.33$ & $1.44E-02$ & $0.30$ \\ 
 & $3.12E-02$ & $1.65E-03$ & $1.32$ & $1.16E-02$ & $0.32$ \\ 
\end{tabular}
\caption{History of convergence of Example \ref{example2}.}\label{table2}
\end{table}



\end{example}

\section*{Acknowledgment} The work of Huangxin Chen was supported by the NSF of China (Grant No. 11201394) 
and the Fundamental Research Funds for the Central Universities (Grant No. 20720150005). The work of Weifeng Qiu 
was partially supported by a grant from the Research Grants Council of the Hong Kong Special Administrative Region, China 
(Project No. CityU 11302014). Manuel Solano was partially supported by CONICYT-Chile through the FONDECYT project No. 1160320 and
BASAL project CMM, Universidad de Chile, by Centro de Investigaci\'on en Ingenier\'ia Matem'atica (CI$^2$MA), Universidad
de Concepci\'on, and by CONICYT project Anillo ACT1118 (ANANUM).
As a convention the names of the authors are alphabetically ordered. All authors 
contributed equally in this article. 

\providecommand{\bysame}{\leavevmode\hbox to3em{\hrulefill}\thinspace}
\providecommand{\MR}{\relax\ifhmode\unskip\space\fi MR }
\providecommand{\MRhref}[2]{%
  \href{http://www.ams.org/mathscinet-getitem?mr=#1}{#2}
}
\providecommand{\href}[2]{#2}

\end{document}